%% file: doc-long.tex
\documentclass[envcountsame,orivec]{llncs}
\usepackage[utf8]{inputenc}
\usepackage[charter]{mathdesign}

\usepackage{hyperref}

\hypersetup{
	pdfborder={0 0 0}
}

\newcommand{\andd}{\wedge}
\newcommand{\orr}{\vee}

\newcommand{\imp}{\rightarrow}

\newcommand{\biimp}{\leftrightarrow}
\newcommand{\Nb}{\mathbb{N}}

\newcommand{\Psf}{\mathsf{P}}
\newcommand{\Qsf}{\mathsf{Q}}

\newcommand{\Ccal}{\mathcal{C}}

\newcommand{\Fcal}{\mathcal{F}}

\newcommand{\Mcal}{\mathcal{M}}
\newcommand{\Pcal}{\mathcal{P}}

\newcommand{\Rcal}{\mathcal{R}}
\newcommand{\Scal}{\mathcal{S}}

\newcommand{\Tcal}{\mathcal{T}}

\newcommand{\uh}{{\upharpoonright}}

\renewcommand{\setminus}{\smallsetminus}

\newcommand{\tuple}[1]{\left\langle #1 \right\rangle}
\newcommand{\subseteqfin}{\subseteq_{\tiny\texttt{fin}}}

\newcommand{\s}[1]{\ensuremath{\sf{#1}}}

\newcommand{\rca}{\s{RCA}_0}

\newcommand{\wkl}{\s{WKL}}
\newcommand{\wwkl}{\s{WWKL}}
\newcommand{\dnr}{\s{DNC}}

\newcommand{\rt}{\s{RT}}
\newcommand{\srt}{\s{SRT}}

\newcommand{\psrt}{\s{psRT}}
\newcommand{\ads}{\s{ADS}}
\newcommand{\sads}{\s{SADS}}

\newcommand{\cac}{\s{CAC}}
\newcommand{\scac}{\s{SCAC}}

\newcommand{\wscac}{\s{WSCAC}}

\newcommand{\coh}{\s{COH}}

\newcommand{\emo}{\s{EM}}

\newenvironment{proof*}[1][\proofname]{
  
  \begin{proof}}{\end{proof}}

\usepackage{xcolor}	
\usepackage{soul}
\definecolor{grey}{rgb}{0.95,0.95,0.95}
\definecolor{green}{rgb}{0.2,0.6,0.4}

\title{Partial orders and immunity in reverse mathematics}
\author{
  Ludovic Patey
}

\institute{
Laboratoire PPS, Universit\'e Paris Diderot, Paris, FRANCE\\
\email{ludovic.patey@computability.fr}
}

\date{\today}

\begin{document}

\maketitle

\begin{abstract}
We identify computability-theoretic properties enabling us to separate
various statements about partial orders in reverse mathematics. We obtain simpler
proofs of existing separations, and deduce new compound ones. This work is part of a larger
program of unification of the separation proofs of various Ramsey-type theorems in reverse mathematics
in order to obtain a better understanding of the combinatorics of Ramsey's theorem and its consequences.
We also answer a question of Murakami, Yamazaki and Yokoyama about pseudo Ramsey's theorem for pairs.
\end{abstract}

\section{Introduction}
\input{parts/part0-introduction-long}

\section{Preservation of properties for co-c.e.\ sets}
\input{parts/part1-preservation-coce-long}

\section{$\cac$ and constant-bound immunity}\label{sect:cac-cb-immunity}
\input{parts/part2-cac-immunity-long}

\section{$\ads$ and pseudo Ramsey's theorem for pairs}
\input{parts/part3-ads-psrt22}

\section{$\ads$ and dependent hyperimmunity}
\input{parts/part4-ads-dependent-hyperimmunity-long}

\section{Weakly stable partial orders}
\input{parts/part5-wscac-scac}

\section{Partial orders with compactness}
\input{parts/part7-questions}

\vspace{0.5cm}

\noindent \textbf{Acknowledgements}. The author is thankful to the reviewers for their numerous suggestions
of improvement.

\vspace{0.5cm}

\bibliographystyle{splncs03}
\bibliography{../bibliography}

\end{document}

%% file: parts/part0-introduction-long.tex
Many theorems of ``ordinary'' mathematics are of the form
$$
(\forall X)[\Phi(X) \imp (\exists Y)\Psi(X,Y)]
$$
where $\Phi$ and $\Psi$ are arithmetic formulas.
They can be seen as \emph{mathematical problems},
whose \emph{instances} are sets~$X$ such that~$\Phi(X)$ holds,
and whose \emph{solutions} to~$X$ are sets~$Y$ such that~$\Psi(X,Y)$ holds.
For example, K\"onig's lemma asserts that every infinite, finitely branching 
tree admits an infinite path through it.
 
There exist many ways to calibrate the strength of a mathematical problem.
Among them, \emph{reverse mathematics} is a vast foundational program that seeks to determine the weakest
axioms necessary to prove ordinary theorems. It uses the framework of subsystems of second-order arithmetic,
within the base theory $\rca$, which can be thought of as capturing \emph{computable mathematics}.
An \emph{$\omega$-structure} is a structure whose first-order part consists of the standard integers.
The $\omega$-models of $\rca$ are those whose second-order part is a \emph{Turing ideal},
that is, a collection of sets~$\Scal$ downward-closed under the Turing reduction and closed under the effective join.

In this setting, an $\omega$-model $\Mcal$ satisfies a mathematical problem~$\Psf$
if every $\Psf$-instance in $\Mcal$ has a solution in~$\Mcal$.
A standard way of proving that a problem~$\Psf$ does not imply another problem~$\Qsf$
consists of creating an $\omega$-model $\Mcal$ satisfying $\Psf$ but not $\Qsf$.
Such a model is usually constructed by taking a ground Turing ideal, and extending it
by iteratively adding solutions to its $\Psf$-instances. However, while taking the
closure of the collection~$\Mcal \cup \{Y\}$ to obtain a Turing ideal, one may add solutions
to $\Qsf$-instances as well. The whole difficulty of this construction consists
of finding the right computability-theoretic notion preserved by $\Psf$ but not by~$\Qsf$.

We conduct a program of identification of the computability-theoretic properties
enabling us to distinguish various Ramsey-type theorems in reverse mathematics,
but also under computable and Weihrauch reducibilities.
This program puts emphasis on the interplay between computability theory and reverse mathematics,
the former providing tools to separate theorems in reverse mathematics over standard models,
and the latter exhibiting new computability-theoretic properties.

Among the theorems studied in reverse mathematics, the ones coming from Ramsey's theory play
a central role. Their strength are notoriously hard to gauge, and required the development
of involved computability-theoretic frameworks. Perhaps the most well-known example
is Ramsey's theorem.

\begin{definition}[Ramsey's theorem]
A subset~$H$ of~$\omega$ is~\emph{homogeneous} for a coloring~$f : [\omega]^n \to k$ (or \emph{$f$-homogeneous}) 
if each $n$-tuple from~$H$ is given the same color by~$f$. 
$\rt^n_k$ is the statement ``Every coloring $f : [\omega]^n \to k$ has an infinite $f$-homogeneous set''.
\end{definition}

Jockusch~\cite{Jockusch1972Ramseys} conducted a computational analysis of Ramsey's theorem.
He proved in particular that $\rt^n_k$ implied the existence of the halting set whenever $n \geq 3$.
There has been a lot of literature around the strength of Ramsey's theorem for pairs~\cite{Cholak2001strength,Dzhafarov2009Ramseys,Hirschfeldt2008strength,Seetapun1995strength}
and its consequences~\cite{Cholak2001Free,Csima2009strength,Hirschfeldt2007Combinatorial}.
We focus on some mathematical statements about partial orders which are consequences
of Ramsey's theorem for pairs.

\begin{definition}[Chain-antichain]
A \emph{chain} in a partial order $(P, \leq_P)$ is a set $S \subseteq P$ such that $(\forall x, y \in S)(x \leq_P y \orr y \leq_P x)$.  An \emph{antichain} in $P$ is a set $S \subseteq P$ such that $(\forall x, y \in S)(x \neq y \imp x |_P y)$ (where $x |_P y$ means that $x \nleq_P y \andd y \nleq_P x$).   $\cac$ is the statement ``every infinite partial order has an infinite chain or an infinite antichain.'' 
\end{definition}

The chain-antichain principle was introduced by Hirschfeldt and Shore~\cite{Hirschfeldt2007Combinatorial}
together with the ascending descending sequence ($\ads$). They studied extensively cohesive and stable versions
of the statements, and proved that $\cac$ is computationally weak, in that it does not even imply the existence
of a diagonally non-computable function. However, their proof has an ad-hoc flavor, in that
it is a direct separation involving the two statements. Later, Lerman, Solomon and Towsner~\cite{Lerman2013Separating}
separated~$\ads$ from~$\cac$ over $\omega$-models by using an involved iterated forcing argument.

In this paper, we revisit the two proofs and emphasize the combinatorial nature of the principles
by identifying the computability-theoretic properties separating them. Those properties happen to be very natural
and coincide on co-c.e.\ sets with some well-known computability-theoretic notions, namely, immunity and hyperimmunity.
The proof of the separation of $\ads$ from~$\cac$ is significantly simpler and more modular,
as advocated by the author in~\cite{Patey2015Iterativea}.
Last, we give a simpler separation of two versions of stability for the chain-antichain principles
over computable reducibility, which was previously proven by Astor et al.~\cite{Astor2016uniform} by the means of a mutually dependent
elaborate notion of forcing.\footnote{
This paper is an extended version of a conference paper of the same name published in CiE 2016.}

\subsection{Notation and definitions}

Given two sets $A$ and $B$, we denote by $A < B$ the formula
$(\forall x \in A)(\forall y \in B)[x < y]$
and by $A \subseteq^{*} B$ the formula $(\forall^{\infty} x \in A)[x \in B]$,
meaning that $A$ is included in $B$ \emph{up to finitely many elements}.
A \emph{Mathias condition} is a pair $(F, X)$
where $F$ is a finite set, $X$ is an infinite set and $F < X$.
A condition $(F_1, X_1)$ \emph{extends } $(F, X)$ (written $(F_1, X_1) \leq (F, X)$)
if $F \subseteq F_1$, $X_1 \subseteq X$ and $F_1 \setminus F \subset X$.
A set $G$ \emph{satisfies} a Mathias condition $(F, X)$
if $F \subset G$ and $G \setminus F \subseteq X$.

%% file: parts/part1-preservation-coce-long.tex
Ramsey's theorem for $k$ colors has a deeply disjunctive nature. One cannot know in a finite amount of time
whether a coloring will admit an infinite homogeneous set for a fixed color, and one must therefore
build multiple homogeneous sets simultaneously, namely, one for each color.
This disjunction was exploited by the author to show for example that
$\ads$ does not preserve 2 hyperimmunities simultaneously, whereas the Erd\H{o}s-Moser theorem
does~\cite{Patey2015Iterativea}. This idea was also used in the context of computable reducibility
to show that~$\rt^2_{k+1}$ does not computably reduce to~$\rt^2_k$ whenever $k \geq 1$,
by showing that~$\rt^2_k$ preserves 2 among $k+1$ hyperimmunities simultaneously
whereas $\rt^2_{k+1}$ does not~\cite{Patey2016weakness}.
In this section, we shall see that this disjunctive flavor disappears whenever
considering co-c.e.\ sets. In particular, $\rt^2_2$ admits preservation
of countably many hyperimmune co-c.e.\ sets simultaneously.

\begin{definition}[Hyperimmunity]
An \emph{array} is a sequence of mutually disjoint finitely coded sets.
An array $F_0, F_1, \dots$ \emph{traces} a set~$A$ if $F_i \cap A \neq \emptyset$ for every~$i \in \omega$.
A set~$A$ is \emph{$X$-hyperimmune} if it is not traced by any $X$-computable array.
\end{definition}

Equivalently, a set is $X$-hyperimmune if its principal function
is not dominated by any $X$-computable function, where the \emph{principal function} 
$p_A$ of a set~$A = \{x_0 < x_1 < \dots \}$ is defined by $p_A(i) = x_i$.

\begin{definition}[Preservation of hyperimmunity for co-c.e.\ sets]
A $\Pi^1_2$ statement~$\Psf$ admits \emph{preservation of hyperimmunity for co-c.e.\ sets} if
for every set $Z$, every sequence of \emph{$Z$-co-c.e.} $Z$-hyperimmune sets~$A_0, A_1, \dots$
and every $\Psf$-instance~$X \leq_T Z$, there is a solution~$Y$ to~$X$ such that the~$A$'s are 
$Y \oplus Z$-hyperimmune.
\end{definition}

The remainder of this section is devoted to the proof that preservation of 1 hyperimmunity
and preservation of (countable) hyperimmunity for co-c.e.\ sets coincide.

\begin{lemma}\label{lem:co-ce-hyp-union-hyp}
If $A_0$ and~$A_1$ are co-c.e.\ hyperimmune sets,
then so is $A_0 \cup A_1$.
\end{lemma}
\begin{proof}
If $A_0$ and $A_1$ are co-c.e., then so is $A_0 \cup A_1$.
We now prove that $A_0 \cup A_1$ is hyperimmune.
Let~$F_0, F_1, \dots$ be an array tracing~$A_0 \cup A_1$.
If for infinitely many~$s$, $F_s \cap A_0 = \emptyset$,
then we can $\vec{F}$-computably find infinitely many such~$s$
since $A_0$ is co-c.e. Note that for any such~$s$, $F_s \cap A_1 \neq \emptyset$.
Since $A_1$ is hyperimmune, it follows that $\vec{F}$ is not computable.
If for all but finitely many~$s$, $F_s \cap A_0 \neq \emptyset$,
then we can $\vec{F}$-compute an array tracing $A_0$,
and therefore $\vec{F}$ is not computable by hyperimmunity of~$A_0$.
\end{proof}

\begin{lemma}\label{lem:co-ce-countable-union-hyp}
Let~$A_0, A_1, \dots$ be a (non-effective) listing of co-c.e.\ hyperimmune sets.
There is a hyperimmune set~$B$ such that every array tracing any~$A_i$
computes an array tracing~$B$.
\end{lemma}
\begin{proof}
Let~$B = \{ \tuple{x, y} : x \in A_0 \wedge y \in \bigcup_{j \leq x} A_j \}$.
For every array $F_0, F_1, \dots$ tracing~$A_n$, and every~$x \in A_0$ such that~$x \geq n$,
the $\vec{F}$-computable array $G_0, G_1, \dots$ defined by $G_i = \{ \tuple{x, y} : y \in F_i \}$ traces~$B$.
We now prove that $B$ is hyperimmune. Let~$F_0, F_1, \dots$ be an array tracing~$B$.
We need to prove that~$\vec{F}$ is not computable.
Let~$G_0, G_1, \dots$ be the $\vec{F}$-computable sequence defined inductively as follows:
$G_0 = \{x_0\}$ for some~$x_0 \in A_0$.
Assume we have defined~$G_i$, and let $x_i = \max G_i$.
We search for stages~$t > s > i$ such that
$$
(\forall \tuple{x, y} \in F_s)[x \leq x_i \rightarrow (x \not \in A_{0,t} \vee y \not \in \bigcup_{j \leq x} A_{j,t})]
$$
where $A_{j,s}$ is the approximation of the co-c.e.\ set $A_j$ at stage $s$.
If we find such a stage, we let $G_{i+1} = \{ x > x_i : \tuple{x, y} \in F_s \}$.
We have two cases.
In the first case, the sequence of the $G$'s is infinite. By construction, $G$ is an $\vec{F}$-computable array tracing
	$A_0$. It follows that $\vec{F}$ is not computable by hyperimmunity of~$A_0$.
In the second case, the sequence of the~$G$'s is finite. Let~$G_i$ be its last element and $x_i = \max G_i$.
	Then, for every $t > s \geq i$, there is some~$\tuple{x, y} \in F_s$,
	such that $x \leq x_i$ and $y \in \bigcup_{j \leq x} A_j$.
	For every~$s \geq i$, let~$H_s = \{ y : (\exists \tuple{x, y} \in F_s)[x \leq x_i] \}$.
	The $\vec{F}$-computable array $\vec{H}$ traces $\bigcup_{j \leq x_i} A_j$.
	However, each~$A_j$ is co-c.e.\ and hyperimmune, so by Lemma~\ref{lem:co-ce-hyp-union-hyp}, $\bigcup_{j \leq x_i} A_j$
	is hyperimmune. It follows that $\vec{F}$ is not computable.
\end{proof}

\begin{lemma}\label{lem:preservation-hyperimmunity-coce}
If $\Psf$ admits preservation of 1 hyperimmunity, then it admits 
preservation of hyperimmunity for co-c.e.\ sets.
\end{lemma}
\begin{proof}
Fix a set~$Z$, a countable sequence of $Z$-co-c.e.\ $Z$-hyperimmune sets $A_0, A_1, \dots$,
and a $Z$-computable $\Psf$-instance~$X$.
By a relativization of Lemma~\ref{lem:co-ce-countable-union-hyp}, there is a $Z$-hyperimmune set~$B$
such that every array tracing any $A_i$ $Z$-computes an array tracing $B$.
By preservation of 1 hyperimmunity of~$\Psf$, there is a solution $Y$ to $X$ such that
$B$ is $Y \oplus Z$-hyperimmune. It follows that each~$A_i$ is $Y \oplus Z$-hyperimmune.
\end{proof}

\begin{corollary}\label{cor:rt22-preservation-hyperimmunity-coce}
$\rt^2_2$ admits preservation of hyperimmunity for co-c.e.\ sets.
\end{corollary}
\begin{proof}
The author proved in~\cite{Patey2015Iterativea} that $\rt^2_2$ admits preservation of 1 hyperimmunity.
The corollary now follows from Lemma~\ref{lem:preservation-hyperimmunity-coce}.
\end{proof}

%% file: parts/part2-cac-immunity-long.tex
Hirschfeldt and Shore~\cite{Hirschfeldt2007Combinatorial}
separated $\cac$ from $\dnr$ in reverse mathematics by a direct construction.
$\dnr$ is the statement asserting, for every set~$X$, the existence of a function~$f$
such that $f(e) \neq \Phi^X_e(e)$ for every~$e$.
In this section, we extract the core of the combinatorics of their forcing argument
to exhibit a computability-theoretic property separating the two notions, namely,
constant-bound immunity.

\begin{definition}[Constant-bound immunity]
A \emph{$k$-enumeration} ($k$-enum) of a set~$A$ is an infinite sequence of $k$-sets~$F_0 < F_1 < \dots$
such that for every~$i \in \omega$, $F_i \cap A \neq \emptyset$.
A \emph{constant-bound enumeration} (c.b-enum) of a set~$A$ is a $k$-enumeration of~$A$ for some~$k \in \omega$.
A set $A$ is \emph{$k$-immune (c.b-immune) relative to~$X$} if it admits no $X$-computable $k$-enumeration (c.b-enumeration).
\end{definition}

In particular, 1-immunity coincides with the standard notion of immunity. 
Also note that one can easily create a c.b-immune set computing no effectively immune set.
The following lemma shows that c.b-immunity and immunity coincide for co-c.e.\ sets.

\begin{lemma}\label{lem:cbimmunity-is-immunity-coce}
An $X$-co-c.e.\ set~$A$ is c.b-immune relative to~$X$ iff it is $X$-immune.
\end{lemma}
\begin{proof}
We first prove that if $A$ is not $X$-immune, then it is not c.b-immune relative to~$X$.
Let~$W = \{ w_0 < w_1 < \dots \}$ be an infinite $X$-computable infinite subset of~$A$.
The sequence $F_0, F_1, \dots$ defined by $F_i = \{ w_i \}$ is an $X$-computable 1-enum
(hence c.b-enum) of $A$. Therefore, $A$ is not c.b-immune relative to~$X$.

We now show by induction over~$k$ if $A$ is $X$-co-c.e.\ and has an $X$-computable $k$-enumeration $F_0, F_1, \dots$
then it has an infinite $X$-computable subset. If~$k = 1$, then it is already an infinite subset of~$A$.
Suppose now that~$k \geq 2$. If there are infinitely many $i \in \omega$
such that~$min(F_i) \in \overline{A} \neq \emptyset$, then since~$A$ is $X$-co-c.e.,
one can find an $X$-computable infinite set $S$ of such $i$'s.
The sequence~$\{F_i \setminus min(F_i) : i \in S \}$ is an $X$-computable $(k-1)$-enumeration of~$A$,
and by induction hypothesis, there is an $X$-computable subset of~$A$.
If there are only finitely many such~$i$'s, then the sequence~$\{min(F_i) : i \in \omega)\}$ is, up to finite  changes,
an infinite $X$-computable subset of~$X$. \qed
\end{proof}

\begin{definition}[Preservation of c.b-immunity]
A $\Pi^1_2$ statement~$\Psf$ admits preservation of c.b-immunity if
for every set $Z$, every set~$A$ which is c.b-immune relative to~$Z$,
and every $\Psf$-instance~$X \leq_T Z$, there is a solution~$Y$ to~$X$ such that $A$ is c.b-immune relative to~$Y \oplus Z$.
\end{definition}

We can easily relate the notion of preservation of c.b-immunity with the existing
notion of constant-bound enumeration avoidance defined by Liu~\cite{Liu2015Cone} to separate
$\rt^2_2$ from~$\wwkl$ over~$\rca$.

\begin{definition}[Constant-bound enumeration avoidance]
A \emph{$k$-enumeration} ($k$-enum) of a class~$\Ccal \subseteq 2^\omega$ is an infinite sequence of $k$-sets~$F_0 < F_1 < \dots$
of strings such that for every~$i \in \omega$, every string $\sigma \in F_i$ is of length $i$, 
and $\Ccal \cap [F_i] \neq \emptyset$, where $[F_i]$ is the clopen set of all sequences extending some string in $F_i$.
A \emph{constant-bound enumeration} (c.b-enum) of $\Ccal$ is a $k$-enumeration of~$\Ccal$ for some~$k \in \omega$.

A $\Pi^1_2$ statement~$\Psf$ admits constant-bound enumeration avoidance if
for every set $Z$, every class~$\Ccal \subseteq 2^\omega$ with no $Z$-computable c.b-enum,
and every $\Psf$-instance~$X \leq_T Z$, there is a solution~$Y$ to~$X$ such that $\Ccal$ has no $Y \oplus Z$-computable c.b-enum.
\end{definition}

\begin{lemma}\label{lem:preservation-immunity-enum-avoidance}
If $\Psf$ admits preservation of c.b-immunity, then it admits constant-bound enumeration avoidance.
\end{lemma}
\begin{proof}
Fix a non-empty class~$\Ccal \subseteq 2^\omega$, and let
$A = \{ \sigma : \Ccal \cap [\sigma] \neq \emptyset \}$.
We claim that the degrees of the c.b-enums of~$A$ and of~$\Ccal$ coincide.
Any c.b-enum of~$\Ccal$ is a c.b-enum of~$A$. 
Conversely, let~$F_0 < F_1 < \dots$ be a c.b-enum of~$\Ccal$.
We can computably thin it out and normalize it into an enumeration $E_0 < E_1 < \dots$
such that $|\sigma| = i$ for every~$\sigma \in E_i$. \qed
\end{proof}

Hirschfeldt and Shore~\cite{Hirschfeldt2007Combinatorial} proved that~$\cac$ is equivalent to
the existence of homogeneous sets for semi-transitive colorings.
A coloring~$f : [\mathbb{N}]^2 \to 2$ is \emph{semi-transitive}
if whenever~$f(x, y) = 1$ and~$f(y, z) = 1$, then~$f(x, z) = 1$ for $x < y < z$.
We shall use this equivalence to prove the following theorem.

\begin{theorem}\label{thm:cac-preservation-cbimmunity}
$\cac$ admits preservation of c.b-immunity.
\end{theorem}
\begin{proof}
Let~$A$ be a set c.b-immune relative to some set~$Z$, and let~$f : [\omega]^2 \to 2$ be a
$Z$-computable semi-transitive coloring.
Assume that there is no infinite $f$-homogeneous set~$H$ such that $A$ is c.b-immune relative to~$H \oplus Z$,
otherwise we are done. 
We will build two infinite sets~$G_0$ and~$G_1$, such that~$G_i$ is $f$-homogeneous for color~$i$ for each~$i < 2$,
and such that $A$ is c.b-immune relative to~$G_i \oplus Z$ for some~$i < 2$.

The construction is done by a variant of Mathias forcing $(F_0, F_1, X)$, 
where~$F_0$ and $F_1$ are finite sets, and $X$ is infinite $Z$-computable set such that~$max(F_0, F_1) < min(X)$.
Moreover, we require that for every~$i < 2$ and every~$x \in X$, $F_i \cup \{x\}$ is $f$-homogeneous for color~$i$.
A condition~$(E_0, E_1, Y)$ \emph{extends} $(F_0, F_1, X)$ if $(E_i, Y)$ Mathias extends $(F_i, X)$
for each~$i < 2$. A pair of sets~$G_0, G_1$ \emph{satisfies} a condition~$c = (F_0, F_1, X)$
if~$G_i$ is $f$-homogeneous for color~$i$ and satisfies the Mathias condition~$(F_i, X)$ for each~$i < 2$.

\begin{lemma}\label{lem:cac-immunity-force-Q}
For every condition~$c = (F_0, F_1, X)$ and every~$i < 2$, there is an extension $(E_0, E_1, Y)$
of~$c$ such that~$|E_i| > |F_i|$.
\end{lemma}
\begin{proof*}
Take any $x \in X$ such that the set~$Y = \{ y \in X : f(x, y) = i \}$ is infinite.
Such an~$x$ must exist, otherwise the set~$X$ is limit-homogeneous for color~$1-i$
and one can $Z$-compute an infinite $f$-homogeneous set, contradicting our hypothesis.
Let~$E_i = F_i \cup \{x\}$ and $E_{1-i} = F_{1-i}$, and take $(E_0, E_1, Y)$ as the desired extension.
\end{proof*}

In what follows, we interpret $\Phi_0, \Phi_1, \dots$
as Turing functionals outputting non-empty finite sets such that if $\Phi^X_e(x)$ and $\Phi^X_e(x+1)$
both halt, $\max(\Phi^X_e(x)) < \min(\Phi^X_e(x+1))$.
We want to satisfy the following requirements for each~$e_0, k_0, e_1, k_1 \in \omega$:
$$
\Rcal_{e_0, k_0, e_1, k_1} : \hspace{10pt} \Rcal^{G_0}_{e_0, k_0} \hspace{10pt} 
		\vee \hspace{10pt} \Rcal^{G_1}_{e_1, k_1}
$$
where $\Rcal_{e,k}^G$ is the requirement 
$$
(\exists x)\left(\Phi^{G \oplus Z}_e(x) \uparrow \vee |\Phi^{G \oplus Z}_e(x)| > k \vee 
 \Phi^{G \oplus Z}_e(x) \cap A = \emptyset\right).
$$
In other words, $\Rcal_{e,k}^G$ asserts that $\Phi^{G \oplus Z}_e$ is not a $k$-enumeration of~$A$.
A condition~$c$ \emph{forces} a formula~$\varphi(G_0, G_1)$ if
$\varphi(G_0, G_1)$ holds for every pair of \emph{infinite} sets~$G_0, G_1$ satisfying~$c$.

\begin{lemma}\label{lem:cac-immunity-force-R}
For every condition~$c$ and every vector of indices~$e_0, k_0, e_1, k_1 \in \omega$,
there is an extension~$d$ of~$c$ forcing~$\Rcal_{e_0, k_0, e_1, k_1}$.
\end{lemma}
\begin{proof*}
Fix a condition~$c = (F_0, F_1, X)$, and
let $P_0, P_1, \dots$ be an $Z$-computable sequence of sets where 
$P_n = \Phi_{e_0}^{(F_0 \cup E_0) \oplus Z}(x_0) \cup \Phi_{e_1}^{(F_1 \cup E_1) \oplus Z}(x_1)$
for a pair of sets~$E_1 < E_0 \subseteq X$ and some~$x_0, x_1 \in \omega$ such that 
$E_0$ is $f$-homogeneous for color~$0$, $E_1 \cup \{y\}$ is $f$-homogeneous for color~$1$ for each~$y \in E_0$,
and for each~$i < 2$, $\max(P_{n-1}) < \min(\Phi_{e_i}^{(F_i \cup E_i) \oplus Z}(x_i))$
and $|\Phi_{e_i}^{(F_i \cup E_i) \oplus Z}(x_i)| \leq k_i$.
We have two cases.

\begin{itemize}
	\item Case 1: the sequence of the~$P$'s is finite and is defined, say to level~$n-1$.
	If there is a pair of infinite sets~$G_0, G_1$ satisfying~$c$ and some~$x_1 \in \omega$ such that
	$\Phi_{e_1}^{G_1 \oplus Z}(x_1) \downarrow$, $\max(P_{n-1}) < \min(\Phi_{e_1}^{G_1 \oplus Z}(x_1))$,
	and $|\Phi_{e_1}^{G_1 \oplus Z}(x_1)| \leq k_1$,
	then let~$E_1 \subseteq G_1$ be such that $F_1 \cup E_1$ is an initial segment of~$G_1$
	for which $\Phi_{e_1}^{(F_1 \cup E_1) \oplus Z}(x_1) \downarrow$.
	The set~$Y = \{ y \in X : E_1 \cup \{y\} \mbox{ is } f\mbox{-homogeneous for color 1 } \}$
	is a superset of~$G_1$, hence is infinite.
	The condition~$d = (F_0, F_1 \cup E_1, Y)$ is an extension of~$c$ forcing $\Rcal_{e_0, k_0}^{G_0}$,
	hence forcing $\Rcal_{e_0, k_0, e_1, k_1}$.
	If there is no such pair of infinite sets~$G_0, G_1$, then the condition~$c$
	already forces $\Rcal_{e_1, k_1}^{G_1}$, hence $\Rcal_{e_0, k_0, e_1, k_1}$.

	\item Case 2: the sequence of the~$P$'s is infinite.
	By c.b-immunity of~$A$ relative to~$Z$, $P_n \cap A = \emptyset$ for some~$n \in \omega$.
	Let~$E_1 < E_0 \subseteq X$ and~$x_0, x_1 \in \omega$ witness the existence of~$P_n$.
	If~$Y_0 = \{ y \in X : E_0 \cup \{y\} \mbox{ is } f\mbox{-homogeneous for color 0 } \}$ is infinite,
	then the condition $(F_0 \cup E_0, F_1, Y_0)$ is an extension of~$c$ forcing $\Rcal_{e_0, k_0}^{G_0}$.
	If $Y_0$ is finite, then for almost every~$y \in X$, there is some~$x_y \in E_0$ such that~$f(x_y, y) = 1$,
	and by transitivity of $f$ for color 1, $E_1 \cup \{y\}$ is $f$-homogeneous for color~1.
	Indeed, $E_1$ is $f$-homogeneous for color~1 and for each~$x \in E_1$, $f(x, x_y) = f(x_y, y) = 1$.
	In this case, $(F_0, F_1 \cup E_1, Y_1)$ is an extension of~$c$ forcing $\Rcal_{e_1, k_1}^{G_1}$,
	for some~$Y_1 =^{*} X$.
	In both cases, there is an extension of~$c$ forcing $\Rcal_{e_0, k_0, e_1, k_1}$.
\end{itemize}
\end{proof*}

Let~$\Fcal = \{c_0, c_1, \dots\}$ be a sufficiently generic filter containing $(\emptyset, \emptyset, \omega)$,
where~$c_s = (F_{0,s}, F_{1,s}, X_s)$. The filter~$\Fcal$ yields a
pair of sets~$G_0, G_1$ defined by~$G_i = \bigcup_s F_{i,s}$.
By Lemma~\ref{lem:cac-immunity-force-Q}, the sets $G_0$ and~$G_1$ are both infinite,
and by Lemma~\ref{lem:cac-immunity-force-R}, the set $A$ is c.b-immune relative to~$G_i \oplus Z$ for some~$i < 2$.
This completes the proof of Theorem~\ref{thm:cac-preservation-cbimmunity}. \qed
\end{proof}

\begin{theorem}\label{thm:dnc-not-preservation-cbimmunity}
$\dnr$ does not admit preservation of c.b-immunity.
\end{theorem}
\begin{proof}
Let~$\mu_{\emptyset'}$ be the modulus function of~$\emptyset'$, that is,
such that~$\mu_{\emptyset'}(x)$ is the minimum stage~$s$ at which~$\emptyset^{'}_s \uh x = \emptyset' \uh x$.
The sketch of the proof is the following:

Computably split $\omega$ into countably many columns~$X_0, X_1, \dots$ of infinite size.
For example, set~$X_i = \{ \tuple{i, n} : n \in \omega \}$ where~$\tuple{\cdot, \cdot}$
is a bijective function from~$\omega^2$ to~$\omega$.
For each~$i$, let $F_i$ be the set of the $\mu_{\emptyset'}(i)$ first elements of~$X_i$.
The sequence~$F_0, F_1, \dots$ is $\emptyset'$-computable.
Assume for now that we have defined a c.e.\ set~$W$ 
such that the $\Delta^0_2$ set~$A = \bigcup_i F_i \setminus W$ is c.b-immune,
and such that $|X_i \cap W| \leq i$. We claim that every DNC function computes an infinite subset of~$A$.

Let~$f$ be any DNC function. By a classical theorem 
about DNC functions (see Bienvenu et al.~\cite{Bienvenu2015logical} for a proof),
$f$ computes a function~$g(\cdot, \cdot, \cdot)$ such that whenever~$|W_e| \leq n$,
then $g(e, n, i) \in X_i \setminus W_e$. For each~$i$, let~$e_i$ be the index of the
c.e.\ set~$W_{e_i} = W \cap X_i$, and let~$n_i = g(e_i, i, i)$.
Since $|X_i \cap W| \leq i$, $|W_{e_i}| \leq i$, hence $n_i = g(e_i, i, i) \in X_i \setminus W_{e_i}$,
which implies $n_i \in X_i \setminus W$. We then have two cases.

\begin{itemize}
	\item Case 1: $n_i \in F_i$ for infinitely many $i$'s. One can $f$-computably find infinitely many of them
	since $\mu_{\emptyset'}$ is left-c.e. and the sequence of the~$n$'s is $f$-computable.
	Therefore, one can $f$-computably find an infinite subset of~$\bigcup_i F_i \setminus W = A$.

	\item Case 2: $n_i \in F_i$ for only finitely many $i$'s. Then the sequence of the~$n_i$'s dominates
	the modulus function $\mu_{\emptyset'}$, and therefore computes the halting set.
	Since the set~$A$ is $\Delta^0_2$, $f$ computes an infinite subset of~$A$.
\end{itemize}

We now detail the construction of the c.e.\ set~$W$. In what follows, interpret $\Phi_e$ as a partial computable
sequence of finite sets such that if $\Phi_e(x)$ and~$\Phi_e(x+1)$ both halt, then $\max(\Phi_e(x)) < \min(\Phi_e(x+1))$.
We need to satisfy the following requirements for each~$e,k \in \omega$:
$$
\Rcal_{e,k} : \hspace{10pt} \left[\Phi_e \mbox{ total } \wedge (\forall i)(\forall^{\infty} x)(\Phi_e(x) \cap X_i = \emptyset)\right]
	\rightarrow (\exists x)\left[|\Phi_e(x)| > k \vee \Phi_e(x) \subseteq W \right]
$$
We furthermore want to ensure that~$|X_i \cap W| \leq i$ for each~$i$.
We can prove by induction over~$k$ that if $\Rcal_{e,\ell}$ is satisfied for each~$\ell \leq k$, then
the set $A = \bigcup_i F_i \setminus W$ admits no computable $k$-enumeration.
The case $k = 1$ is trivial, since if~$\Phi_e$ is total and has an infinite intersection with~$X_i$ for some~$i \in \omega$,
then it intersects $X_i \setminus F_i$, hence intersects $\overline{A}$.
For the case $k \geq 2$, assume that $\Phi_e$ is total, and has infinite intersection with~$X_i$ for some~$i \in \omega$.
By our assumption that $\max(\Phi_e(x)) < \min(\Phi_e(x+1))$, for large enough $n$,
$F_i < \Phi_e(n) \subseteq X_i$, and hence $\Phi_e(n) \subseteq \overline{A}$.
Otherwise, one can compute a $(k-1)$-enumeration $E_0 < E_1 < \dots$ of~$A$ 
by setting $E_n = \Phi_e(n) \setminus X_i$, and apply the induction hypothesis.

We now explain how to satisfy $\Rcal_{e,k}$ for each~$e, k \in \omega$.
For each pair of indices~$e,k \in \omega$, let~$i_{e,k} = \sum_{\tuple{e',k'} \leq \tuple{e,k}} k'$.
A strategy for~$\Rcal_{e,k}$ \emph{requires attention} at stage~$s > \tuple{e,k}$ if there is an $x < s$
such that $\Phi_{e,s}(x) \downarrow$,
$|\Phi_{e,s}(x)| \leq k$, and $\Phi_{e,s}(x) \subseteq \bigcup_{j \geq i_{e,k}} X_j$. Then, the strategy enumerates all the elements
of~$\Phi_{e,s}$ in~$W$, and is declared  \emph{satisfied}, and will never require attention again.
First, notice that if $\Phi_e$ is total, outputs $k$-sets, and meets finitely many times each~$X_i$, 
then it will require attention at some stage~$s$ and will be declared satisfied. Therefore each requirement~$\Rcal_{e,k}$ is satisfied.
Second, suppose for the sake of contradiction that $|X_i \cap W| > i$ for some~$i$. Let~$s$ be the stage at which it happens,
and let~$\tuple{e,k} < s$ be the maximal pair such that~$\Rcal_{e,k}$ has enumerated some element of~$X_i$ in~$W$. 
In particular, $i_{e,k} \leq i$.
Since the strategy for $\Rcal_{e',k'}$ enumerates at most~$k'$ elements in~$W$, 
$$
\sum_{\tuple{e',k'} \leq \tuple{e,k}} k' \geq |X_i \cap W| > i \geq i_{e,k} = \sum_{\tuple{e',k'} \leq \tuple{e,k}} k'
$$
Contradiction.
\end{proof}

\begin{corollary}[Hirschfeldt and Shore~\cite{Hirschfeldt2007Combinatorial}]
$\rca \wedge \cac \nvdash \dnr$.
\end{corollary}

%% file: parts/part3-ads-psrt22.tex
In this section, we answer a question of Murakami, Yamazaki and Yokoyama in~\cite{Murakami2014Ramseyan}
by proving the equivalence between the ascending descending sequence,
introduced by Hirschfeldt and Shore~\cite{Hirschfeldt2007Combinatorial} and pseudo Ramsey's theorem for pairs and two colors.
This equivalence was independently obtained by Steila in~\cite{SteilaReverse}.

\begin{definition}[Ascending descending sequence]
Given a linear order~$(L, <_L)$, an \emph{ascending} (\emph{descending}) sequence
is a set~$S$ such that for every~$x <_\Nb y \in S$, $x <_L y$ ($x >_L y$).
$\ads$ is the statement ``Every infinite linear order admits an infinite ascending or descending sequence''.
\end{definition}

Pseudo Ramsey's theorem for pairs was first introduced by Friedman~\cite{Friedman2010Adjacent}
and later studied by Friedman and Pelupessy~\cite{Friedman2016Independence},
and Murakami, Yamazaki and Yokoyama in~\cite{Murakami2014Ramseyan} 
who proved that it is between the chain antichain principle
and the ascending descending sequence principle over~$\rca$.

\begin{definition}[Pseudo Ramsey's theorem]
A set~$H = \{x_0 < x_1 < \dots \}$ is \emph{pseudo-homogeneous} for a coloring~$f : [\mathbb{N}]^n \to k$
if~$f(x_i, \dots, x_{i+n-1}) = f(x_j, \dots, x_{j+n-1})$ for every~$i, j \in \mathbb{N}$.
$\mathsf{psRT}^n_k$ is the statement ``Every coloring~$f : [\mathbb{N}]^n \to k$ has an infinite pseudo-homogeneous set''.
\end{definition}

We now prove the equivalence between the two statements.
Note however that the exact strength of $\mathsf{psRT}^2_k$
remains open in reverse mathematics whenever $k > 2$.

\begin{theorem}\label{thm:psrt22-is-ads}
$\mathsf{RCA}_0 \vdash \mathsf{psRT}^2_2 \leftrightarrow \ads$
\end{theorem}
\begin{proof}
The direction~$\mathsf{psRT}^2_2 \rightarrow \ads$ is Theorem~24 in~\cite{Murakami2014Ramseyan}.
We prove that~$\mathsf{ADS} \rightarrow \mathsf{psRT}^2_2$. 
Let~$f : [\mathbb{N}]^2 \to 2$ be a coloring. The reduction is in two steps.
We first define a $\Delta^{0,f}_1$ semi-transitive coloring
$g : [\mathbb{N}]^2 \to 2$ such that every infinite set pseudo-homogeneous for~$g$
computes an infinite set pseudo-homogeneous for~$f$.
Then, we define a $\Delta^{0,g}_1$ linear order~$h : [\mathbb{N}]^2 \to 2$
such that every infinite set pseudo-homogeneous for~$h$
computes an infinite set pseudo-homogeneous for~$g$.
We conclude by applying~$\ads$ over~$h$.

\emph{Step 1}:
Define the coloring~$g : [\mathbb{N}]^2 \to 2$ for every~$x < y$ by
$g(x, y) = 1$ if there exists a sequence~$x = x_0 < \dots < x_l = y$
such that $f(x_i, x_{i+1}) = 1$ for every~$i < l$, and~$g(x, y) = 0$ otherwise.
The function~$g$ is a semi-transitive coloring. Indeed,
suppose that~$g(x, y) = 1$ and~$g(y, z) = 1$, witnessed respectively
by the sequences~$x = x_0 < \dots < x_m = y$ and~$y = y_0 < \dots < y_n = z$.
The sequence~$x = x_0 < \dots < x_m = y_0 < \dots < y_n = z$ witnesses~$g(x, z) = 1$.
We claim that every infinite set~$H = \{x_0 < x_1 < \dots \}$ pseudo-homogeneous for~$g$ computes
an infinite set pseudo-homogeneous for~$f$. If $H$ is pseudo-homogeneous with color~0,
then $f(x_i, x_{i+1}) = 0$ for each~$i$, otherwise the sequence~$x_i < x_{i+1}$ would witness
$g(x_i, x_{i+1}) = 1$. Thus $H$ is pseudo-homogeneous for~$f$ with color~0. 
If $H$ is pseudo-homogeneous with color~1,
then define the set~$H_1 \supseteq H$ to be the set of integers in the sequences
witnessing~$g(x_i, x_{i+1}) = 1$ for each~$i$. The set~$H_1$ is~$\Delta^{0,f \oplus H}_1$
and pseudo-homogeneous for~$f$ with color~1.

\emph{Step 2}:
Define the coloring~$h : [\mathbb{N}]^2 \to 2$ for every~$x < y$ by
$h(x, y) = 0$ if there exists a sequence~$x = x_0 < \dots < x_l = y$
such that $g(x_i, x_{i+1}) = 0$ for every~$i < l$, and~$h(x, y) = 1$ otherwise.
For the same reasons as for~$g$, $h(x, z) = 0$ whenever~$h(x, y) = 0$ and~$h(y, z) = 0$ for~$x < y < z$.
We need to prove that if~$h(x, z) = 0$ then either $h(x, y) = 0$ or~$h(y, z) = 0$ for~$x < y < z$.
Let~$x = x_0 < \dots < x_l = z$ be a sequence witnessing $h(x,z) = 0$.
If~$y = x_i$ for some~$i < l$ then the sequence~$x = x_0 < \dots < x_i = y$ witnesses~$h(x,y) = 0$.
If~$y \neq x_i$ for every~$i < l$, then there exists some~$i < l$ such that~$x_i < y < x_{i+1}$. 
By semi-transitivity of~$g$, either $g(x_i, y) = 0$ or~$g(y, x_{i+1}) = 0$. In this case either~$x = x_0 < \dots < x_i < y$ witnesses~$h(x,y) = 0$ or~$y < x_{i+1} < \dots < x_l = z$ witnesses~$h(y,z) = 0$.
Therefore~$h$ is a linear order.
For the same reasons as for~$g$, every infinite set pseudo-homogeneous for~$h$
computes an infinite set pseudo-homogeneous for~$g$.
This last step finishes the proof. \qed
\end{proof}

%% file: parts/part4-ads-dependent-hyperimmunity-long.tex
Lerman, Solomon and Towsner~\cite{Lerman2013Separating} separated the ascending descending
sequence principle from a stable version of~$\cac$ by using a very involved iterated forcing argument.
According to our previous simplification of their general framework~\cite{Patey2015Iterativea},
we reformulate their proof in terms of preservation of dependent hyperimmunity,
and extend it to pseudo Ramsey's theorem for pairs.

\begin{definition}[Dependent hyperimmunity]
A formula~$\varphi(U, V)$ is \emph{essential}
if for every~$x \in \omega$, there is a finite set~$R > x$ such that for every~$y \in \omega$,
there is a finite set~$S > y$ such that~$\varphi(R, S)$ holds.
A pair of sets~$A_0, A_1 \subseteq \omega$ is \emph{dependently $X$-hyperimmune} 
if for every essential $\Sigma^{0,X}_1$ formula~$\varphi(U, V)$,
$\varphi(R, S)$ holds for some~$R \subseteq \overline{A}_0$ and~$S \subseteq \overline{A}_1$.
\end{definition}

In particular, if the pair $A_0, A_1$ is dependently hyperimmune,
then $A_0$ and~$A_1$ are both hyperimmune.

\begin{definition}[Preservation of dependent hyperimmunity]
A $\Pi^1_2$ statement~$\Psf$ admits \emph{preservation of dependent hyperimmunity} if
for every set $Z$, every pair of dependently $Z$-hyperimmune sets~$A_0, A_1 \subseteq \omega$
and every $\Psf$-instance~$X \leq_T Z$, there is a solution~$Y$ to~$X$ such that $A_0, A_1$ are dependently 
$Y \oplus Z$-hyperimmune.
\end{definition}

A partial order $(P, \leq_P)$ is \emph{stable} if either $(\forall i \in P)(\exists s)[(\forall j > s)(j \in P \imp i \leq_P j) \orr (\forall j > s)(j \in P \imp i \mid_P j)]$ or $(\forall i \in P)(\exists s)[(\forall j > s)(j \in P \imp i \geq_P j) \orr (\forall j > s)(j \in P \imp i \mid_P j)]$. $\scac$ is the restriction of $\cac$ to stable partial orders.
A simple finite injury priority argument shows that~$\scac$ does not admit preservation of dependent hyperimmunity.

\begin{theorem}\label{thm:stable-semi-transitive-fairness}
There exists a computable, stable semi-transitive coloring~$f : [\omega]^2 \to 2$
such that the pair~$\overline{A}_0, \overline{A}_1$ is %dependently 
dep.~hyperimmune, where
$A_i = \{ x : \lim_s f(x, s) = i \}$.
\end{theorem}
\begin{proof}
Fix an enumeration~$\varphi_0(U,V), \varphi_1(U,V), \dots$ of all $\Sigma^0_1$ formulas.
The construction of the function~$f$ is done by a finite injury priority argument
with a movable marker procedure. We want to satisfy the following scheme of requirements for each~$e$,
where~$A_i = \{ x : \lim_s f(x, s) = i \}$:
\[
\Rcal_e : \varphi_e(U,V) \mbox{ essential} \imp 
	(\exists R \subseteq_{fin} A_0)(\exists S \subseteq_{fin} A_1)\varphi_e(R,S)
\]

The requirements are given the usual priority ordering.
We proceed by stages, maintaining two sets $A_0, A_1$
which represent the limit of the function~$f$.
At stage 0, $A_{0,0} = A_{1,0} = \emptyset$ and $f$ is nowhere defined. 
Moreover, each requirement~$\Rcal_e$ is given a movable marker~$m_e$ initialized to~0.

A strategy for~$\Rcal_e$ \emph{requires attention at stage~$s+1$}
if $\varphi_e(R, S)$ holds for some~$R < S \subseteq (m_e, s]$.
The strategy sets~$A_{0,s+1} = (A_{0,s} \setminus (m_e, min(S)) \cup [min(S), s]$ 
and~$A_{1,s+1} = (A_{1,s} \setminus [min(S), s]) \cup (m_e, min(S))$.
Note that~$R \subseteq (m_e, min(S))$ since~$R < S$.
Then it is declared~\emph{satisfied} and does not act until some strategy of higher priority changes its marker. 
Each marker~$m_{e'}$ of strategies of lower priorities is assigned the value~$s+1$.

At stage~$s+1$, assume that~$A_{0,s} \cup A_{1,s} = [0,s)$
and that~$f$ is defined for each pair over~$[0,s)$.
For each~$x \in [0,s)$, set~$f(x,s) = i$ for the unique~$i$ such that~$x \in A_{i,s}$.
If some strategy requires attention at stage~$s+1$, take the least one
and satisfy it.
If no such requirement is found, set~$A_{0,s+1} = A_{0,s} \cup \{s\}$
and~$A_{1,s+1} = A_{1,s}$. Then go to the next stage. This ends the construction.

Each time a strategy acts, it changes the markers of strategies of lower priority, and is declared satisfied.
Once a strategy is satisfied, only a strategy of higher priority can injure it.
Therefore, each strategy acts finitely often and the markers stabilize.
It follows that the $A$'s also stabilize and that~$f$ is a stable function.

\begin{claim}
For every~$x < y < z$, $f(x,y) = 1 \wedge f(y,z) = 1 \imp f(x,z) = 1$
\end{claim}
\begin{proof*}
Suppose that $f(x,y) = 1$ and~$f(y,z) = 1$ but $f(x,z) = 0$. 
By construction of~$f$, $x \in A_{0,z}$, $x \in A_{1,y}$ and~$y \in A_{1,z}$.
Let~$s \leq z$ be the last stage such that~$x \in A_{1,s}$.
Then at stage~$s+1$, some strategy~$\Rcal_e$ receives attention
and moves~$x$ to~$A_{0,s+1}$ and therefore moves~$[x, s]$ to~$A_{0,s+1}$.
In particular~$y \in A_{0,s+1}$ since~$y \in [x,s]$.
Moreover, the strategies of lower priority have had their marker moved to~$s+1$
and therefore will never move any element below~$s$.
Since $f(y, z) = 1$, then $y \in A_{1,z}$.
In particular, some strategy~$\Rcal_i$ of higher priority moved~$y$ to~$A_{1,t+1}$
at stage~$t +1$ for some~$t \in (s, z)$. Since~$\Rcal_i$ has a higher priority, $m_i \leq m_e$, 
and since~$y$ is moved to~$A_{1,t+1}$, then so is~$[m_i, y]$, and in particular
$x \in A_{1,t+1}$ since~$m_i \leq m_e \leq x \leq y$. This contradicts the maximality of~$s$.
\end{proof*}

\begin{claim}
For every~$e \in \omega$, $\Rcal_e$ is satisfied.
\end{claim}
\begin{proof*}
By induction over the priority order. Let~$s_0$ be a stage after which
no strategy of higher priority will ever act. By construction, $m_e$ will not change after stage~$s_0$.
If $\varphi_e(U,V)$ is essential, then $\varphi_e(R,S)$ holds for two sets~ $m_e < R < S$.
Let~$s = 1+max(s_0, S)$. The strategy~$\Rcal_e$ will require attention at some stage before~$s$,
will receive attention, be satisfied and never be injured.
\end{proof*}

This last claim finishes the proof. \qed
\end{proof}

\begin{corollary}\label{cor:scac-not-preservation-of-dependent-hyperimmunity}
$\scac$ does not admit preservation of dependent hyperimmunity.
\end{corollary}
\begin{proof}
Let~$f : [\omega]^2 \to 2$ be the coloring of Theorem~\ref{thm:stable-semi-transitive-fairness}.
By construction, the pair $\overline{A}_0$, $\overline{A}_1$ is dependently hyperimmune, 
where~$A_i = \{ x : \lim_s f(x, s) = i \}$.
Let~$H$ be an infinite $f$-homogeneous set. In particular, $H \subseteq A_0$ or~$H \subseteq A_1$.
We claim that the pair~$\overline{A}_0, \overline{A}_1$ is not dependently~$H$-hyperimmune.
The $\Sigma^{0,H}_1$ formula~$\varphi(U, V)$ defined by~$U \neq \emptyset \wedge V \neq \emptyset \wedge U \cup V \subseteq H$
is essential since~$H$ is infinite. However, if there is some~$R \subseteq A_1$ and~$S \subseteq A_0$
such that~$\varphi(R, S)$ holds, then~$H \cap A_0 \neq \emptyset$ and~$H \cap A_1 \neq \emptyset$,
contradicting the choice of~$H$. Therefore $\overline{A}_0, \overline{A}_1$ is not dependently~$H$-hyperimmune.
Hirschfeldt and Shore~\cite{Hirschfeldt2007Combinatorial} proved that
$\scac$ is equivalent to stable semi-transitive Ramsey's theorem for pairs
over~$\rca$. Therefore~$\scac$ does not admit preservation of dependent hyperimmunity. \qed
\end{proof}

We will now prove the positive preservation result.

\begin{theorem}\label{thm:psrt2k-preservation-dependent-hyperimmunity}
For every~$k \geq 2$, $\psrt^2_k$ admits preservation of dep.~hyperimmunity.
\end{theorem}
\begin{proof}
The proof is done by induction over~$k \geq 2$.
Fix a pair of sets~$A_0, A_1 \subseteq \omega$ dependently $Z$-hyperimmune for some set~$Z$.
Let~$f : [\omega]^2 \to k$ be a $Z$-computable coloring
and suppose that there is no infinite set~$H$ over which~$f$ avoids at least one color, and such that
the pair $A_0, A_1$ is dependently $H \oplus Z$-hyperimmune,
as otherwise, we are done by induction hypothesis.
We will build $k$ infinite sets~$G_0, \dots, G_{k-1}$
such that $G_i$ is pseudo-homogeneous for~$f$ with color~$i$ for each~$i < k$
and such that $A_0, A_1$ is dependently $G_i \oplus Z$-hyperimmune for some~$i < k$.
The sets~$G_0, \dots, G_{k-1}$ are built by a variant of Mathias forcing
$(F_0, \dots, F_{k-1}, X)$ such that
\begin{itemize}
	\item[(i)] $F_i \cup \{x\}$ is pseudo-homogeneous for~$f$ with color~$i$ for each~$x \in X$
	\item[(ii)] $X$ is an infinite set such that $A_0, A_1$ is dependently $X \oplus Z$-hyperimmune
\end{itemize}
A condition~$d = (H_0, \dots, H_{k-1}, Y)$ \emph{extends} $c = (F_0, \dots, F_{k-1}, X)$
(written $d \leq c$) if~$(H_i, Y)$ Mathias extends $(F_i, X)$ for each~$i < k$.
A tuple of sets~$G_0, \dots, G_{k-1}$ \emph{satisfies} $c$ if
for every~$n \in \omega$, there is an extension~$d = (H_0, \dots, H_{k-1}, Y)$ of~$c$
such that~$G_i \uh n \subseteq H_i$ for each~$i < k$.
Informally, $G_0, \dots, G_{k-1}$ satisfy~$c$ if the sets are generated by a decreasing
sequence of conditions extending~$c$.
In particular, $G_i$ is pseudo-homogeneous for~$f$ with color~$i$
and satisfies the Mathias condition~$(F_i, X)$. 
The first lemma shows that every sufficiently generic filter yields a $k$-tuple of infinite sets.

\begin{lemma}\label{lem:psrt-cac-fairness-force-Q}
For every condition~$c = (F_0, \dots, F_{k-1}, X)$ and every~$i < k$,
there is an extension~$d = (H_0, \dots, H_{k-1}, Y)$ of~$c$ such that~$|H_i| > |F_i|$.
\end{lemma}
\begin{proof}
Fix~$c$ and $i < k$. If for every $x \in X$ and for all but finitely many $y \in X$,
$f(x, y) \neq i$, then we could $X$-computably thin out the set $X$
to obtain an infinite set $H$ over which $f$ avoids at least one color, contradicting our initial assumption.
Therefore there must be some $x \in X$ such that the set $Y = \{ y \in X : y > x \wedge f(x, y) = i \}$ is infinite.
The condition $d = (F_0, \dots, F_{i-1}, F_i \cup \{x\}, F_{i+1}, \dots, F_{k-1}, Y)$ is the desired extension of~$c$.
\end{proof}

Fix an enumeration~$\varphi_0(G, U,V), \varphi_1(G, U,V), \dots$ of all $\Sigma^{0,Z}_1$ formulas.
We want to satisfy the following requirements for each~$e_0, \dots, e_{k-1} \in \omega$:
$$
\Rcal_{\vec{e}} : \hspace{10pt} \Rcal^{G_0}_{e_0} \hspace{10pt} 
		\vee \hspace{10pt} \dots \hspace{10pt} \vee \hspace{10pt} \Rcal^{G_{k-1}}_{e_{k-1}}
$$
where $\Rcal_e^G$ is the requirement ``$\varphi_e(G, U, V)$ essential $\imp \varphi_e(G, R, S)$ 
for some~$R \subseteq \overline{A}_0$ and~$S \subseteq \overline{A}_1$''.
We say that a condition~$c$ \emph{forces} $\Rcal_{\vec{e}}$ if $\Rcal_{\vec{e}}$  holds
for every $k$-tuple of sets satisfying~$c$. Note that the notion of satisfaction 
has a precise meaning given above.

\begin{lemma}\label{lem:psrt-cac-fairness-force-R}
For every condition~$c$ and every $k$-tuple of indices~$e_0, \dots, e_{k-1} \in \omega$,
there is an extension~$d$ of~$c$ forcing~$\Rcal_{\vec{e}}$.
\end{lemma}
\begin{proof*}
Fix a condition~$c = (F_0, \dots, F_{k-1}, X)$. 
Let~$\psi(U, V)$ be the~$\Sigma^{0,X \oplus Z}_1$ formula which holds if
there is a $k$-tuple of sets~$E_0, \dots, E_{k-1} \subseteq X$ and a $z \in X$ such that for each~$i < k$,
\begin{itemize}
	\item[(i)] $z > max(E_i)$
	\item[(ii)] $F_i \cup E_i \cup \{z\}$ is pseudo-homogeneous for color~$i$.
	\item[(iii)] $\varphi_{e_i}(F_i \cup E_i, U_i, V_i)$ holds for some~$U_i \subseteq U$ and~$V_i \subseteq V$
\end{itemize}
Suppose that~$c$ does not force~$\Rcal_{\vec{e}}$, otherwise we are done.

We claim that~$\psi$ is essential. 
Since~$c$ does not force~$\Rcal_{\vec{e}}$, there is a $k$-tuple of infinite sets~$G_0, \dots, G_{k-1}$
satisfying~$c$ and such that $\varphi_{e_i}(G_i, U, V)$ is essential for each~$i < k$.
Fix some $x \in \omega$. By definition of being essential, there are some finite sets
$R_0, \dots, R_{k-1} > x$ such that for every~$y \in \omega$, there are finite sets~$
S_0, \dots, S_{k-1} > y$ such that~$\varphi_{e_i}(G_i, R_i, S_i)$ holds for each~$i < k$.
Let~$R = \bigcup R_i$ and fix some~$y \in \omega$. There are 
finite sets~$S_0, \dots, S_{k-1} > y$ such that~$\varphi_{e_i}(G_i, R_i, S_i)$ holds for each~$i < k$.
Let~$S = \bigcup S_i$. By continuity, there are finite sets~$E_0, \dots, E_{k-1}$ such that
$G_i \uh max(E_i) = F_i \cup E_i$ and~$\varphi_{e_i}(F_i \cup E_i, R_i, S_i)$ holds for each~$i < k$.
By our precise definition of satisfaction, we can even assume without loss of generality
that $(F_0 \cup E_0, \dots, F_{k-1} \cup E_{k-1}, Y)$ is a valid extension of~$c$ for some infinite set~$Y \subseteq X$.
Let~$z \in Y$. In particular, by the definition of being a condition extending~$c$, 
$z \in X$,  $z > max(E_0, \dots, E_{k-1})$
and $F_i \cup E_i \cup \{z\}$ is pseudo-homogeneous for color~$i$ for each~$i < k$.
Therefore~$\psi(R, S)$ holds, as witnessed by~$E_0, \dots, E_{k-1}$ and~$z$.
Thus $\psi(R, S)$ is essential.

Since $A_0, A_1$ is dependently $X \oplus Z$-hyperimmune, then $\psi(R, S)$ holds for some~$R \subseteq \overline{A}_0$
and some~$S \subseteq \overline{A}_1$. Let~$E_0, \dots, E_{k-1} \subseteq X$
be the $k$-tuple of sets and~$z \in X$ be the integer witnessing $\psi(R,S)$.
Let~$i < k$ be such that the set~$Y = \{w \in X \setminus [0, max(E_i)] : f(z, w) = i\}$ is infinite.
The condition~$d = (F_0, \dots, F_{i-1}, F_i \cup E_i \cup \{z\}, F_{i+1}, \dots, F_{k-1}, Y)$ is
a valid extension of~$c$ forcing~$\Rcal_{\vec{e}}$.
\end{proof*}

Let~$\Fcal = \{c_0, c_1, \dots\}$ be a sufficiently generic filter containing $(\emptyset, \dots, \emptyset, \omega)$,
where $c_s = (F_{0,s}, \dots, F_{k-1,s}, X_s)$. The filter~$\Fcal$ yields a
$k$-tuple of sets~$G_0, \dots, G_{k-1}$ defined by~$G_i = \bigcup_s F_{i,s}$.
By construction, $G_0, \dots, G_{k-1}$ satisfies every condition in~$\Fcal$.
By Lemma~\ref{lem:psrt-cac-fairness-force-Q}, the set~$G_i$ is infinite for each~$i < k$
and by Lemma~\ref{lem:psrt-cac-fairness-force-R}, the pair $A_0, A_1$ is dependently $G_i \oplus Z$-hyperimmune
for some~$i < k$. \qed
\end{proof}

\begin{theorem}\label{thm:randomness-dependent-hyperimmunity}
Fix some set~$Z$ and a pair of sets~$A_0, A_1$ dependently $Z$-hyperimmune.
If~$Y$ is sufficiently random relative to~$Z$, then the pair $A_0, A_1$ is dependently $Y \oplus Z$-hyperimmune.
\end{theorem}
\begin{proof}
It suffices to prove that for every~$\Sigma^{0,Z}_1$ formula~$\varphi(G, U, V)$
and every~$i \in \omega$, the following class is Lebesgue null.
\[
\Scal = \{X : [\varphi(X,U,V) \mbox{\itshape\ is essential }] \wedge 
	(\forall R, S \subseteqfin \omega)\varphi(X, R, S) \imp R \not \subseteq \overline{A}_0 \vee S \not \subseteq \overline{A}_1 \}
\]
Suppose it is not the case. There exists~$\sigma \in 2^{<\omega}$ such that
\[
\mu\{ X \in \Scal : \sigma \prec X \} > 0.8 \cdot 2^{-|\sigma|}
\]
Define
\[
\psi(U, V) = [\mu\{ X \succ \sigma : (\exists \tilde{U} \subseteq U)
	(\exists \tilde{V} \subseteq V)\varphi(X, \tilde{U}, \tilde{V})\} > 0.6 \cdot 2^{-|\sigma|}]
\]
By compactness, the formula~$\psi(U, V)$ is~$\Sigma^{0,Z}_1$. 

\begin{claim}%\label{thm:randomness-dependent-hyperimmunity-essential}
$\psi(U, V)$ is essential.
\end{claim}
\begin{proof*}
Suppose it is not. Then, there exists some~$x \in \omega$,
such that for every~$n \in \omega$, there is some~$y_n \in \omega$ such that $\psi([x,n], [y_n, +\infty))$
does not hold. Let~$\Pcal(X, n, y_n)$ be the formula 
$$(\forall \tilde{U} \subseteq [x, n])
	(\forall \tilde{V} \subseteq [y_n, +\infty))\neg \varphi(X, \tilde{U}, \tilde{V})
$$
Unfolding the definition of $\neg \psi([x,n], [y_n, +\infty))$,
\[
\mu\{ X \succ \sigma : \Pcal(X, n, y_n) \} > 0.4 \cdot 2^{-|\sigma|}
\]
Then, by Fatou's lemma,
\[
\mu\{ X \succ \sigma : (\exists^{\infty} n) \Pcal(X, n, y_n) \} > 0.2 \cdot 2^{-|\sigma|}
\]
Since whenever~$\Pcal(X, n, y_n)$ holds, so does $\Pcal(X, n-1, y_n)$,
\[
\mu\{ X \succ \sigma : (\forall n)(\exists y) \Pcal(X, n, y) \} > 0.2 \cdot 2^{-|\sigma|}
\]
Therefore
\[
\mu\{ X \succ \sigma : \varphi(X, U, V) \mbox{ is essential } \} \leq 0.8 \cdot 2^{-|\sigma|}
\]
Contradicting our assumption. This finishes the lemma.
\end{proof*}

By our claim and by dependent $Z$-hyperimmunity of~$A_0, A_1$, there exists some finite sets~$R \subseteq \overline{A}_0$
and $S \subseteq \overline{A}_1$ such that~$\psi(R, S)$ holds. 
For every $R, S$ such that~$\psi(R, S)$ holds,
there exists some~$X \in \Scal$ and some~$\tilde{R} \subseteq R$ and $\tilde{S} \subseteq S$
such that~$\varphi(X,\tilde{R}, \tilde{S})$ holds. By definition of~$X \in \Scal$, $\tilde{R} \not \subseteq \overline{A}_0$
or $\tilde{S} \not \subseteq \overline{A}_1$ and therefore
either $R \not \subseteq \overline{A}_0$ or $S \not \subseteq \overline{A}_1$. Contradiction. \qed
\end{proof}

\begin{corollary}\label{cor:wwkl-preservation-dependent-hyperimmunity}
$\wwkl$ admits preservation of dependent hyperimmunity.
\end{corollary}
\begin{proof}
Let~$Z$ be a set and $A_0, A_1$ be a pair of dependently $Z$-hyperimmune sets.
Fix a $Z$-computable tree of positive measure~$T \subseteq 2^{<\omega}$.
By Theorem~\ref{thm:randomness-dependent-hyperimmunity}, 
the pair~$A_0, A_1$ is dependently $Y \oplus Z$-hyperimmune
for some Martin-L\"of random~$Y$ relative to~$Z$.
By Ku\v{c}era~\cite{Kuvcera1985Measure}, $Y$ is, up to finite prefix, a path through~$T$. \qed
%By Ku\v{c}era~\cite{Kuvcera1985Measure}, $Y$ is, up to finite prefix, a path through~$T$. \qed
\end{proof}

\begin{corollary}
For every~$k \geq 2$, $\rca \wedge \psrt^2_k \wedge \wwkl \nvdash \scac$.
\end{corollary}
\begin{proof}
Immediate by Theorem~\ref{thm:psrt2k-preservation-dependent-hyperimmunity},
Corollary~\ref{cor:wwkl-preservation-dependent-hyperimmunity}, and
Corollary~\ref{cor:scac-not-preservation-of-dependent-hyperimmunity}. \qed
\end{proof}

\begin{corollary}
$\rca \wedge \ads \wedge \wwkl \nvdash \scac$
\end{corollary}
\begin{proof}
Immediate by the previous corollary and Theorem 24 of~\cite{Murakami2014Ramseyan}. \qed
\end{proof}

Whenever requiring the sets~$A_0$ and~$A_1$ to be co-c.e., we recover
the standard notion of hyperimmunity.
Therefore, the restriction of the preservation of dependent hyperimmunity
to co-c.e.\ sets is not a good computability-theoretic property to distinguish
consequences of Ramsey's theorem for pairs.

\begin{lemma}\label{lem:dependent-hyperimmunity-for-coce}
Fix two sets~$A_0, A_1$ such that~$A_0$ is $X$-co-c.e.
The pair~$A_0, A_1$ is dependently $X$-hyperimmune iff $A_0$ and $A_1$ are $X$-hyperimmune.
\end{lemma}
\begin{proof}
We first show that if~$A_0$ and~$A_1$ are dependently $X$-hyperimmune
then both~$A_0$ and~$A_1$ are $X$-hyperimmune. Let~$F_0, F_1, \dots$ be a $X$-c.e.\ array.
Let $\varphi(U, V)$ be the $\Sigma^{0,X}_1$ formula which holds if~$U = F_i$ for some~$i \in \omega$.
The formula $\varphi(U, V)$ is essential, therefore there $\varphi(R, S)$ holds for some finite set~$R \subseteq \overline{A}_0$
and $S \subseteq \overline{A}_1$. In particular, $R = F_i$ for some~$i \in \omega$, therefore $F_i \subseteq \overline{A}_0$
and $A_0$ is hyperimmune. Similarly, the $\Sigma^{0,X}_1$ formula $\psi(U, V)$ which holds if $V = F_i$ for some~$i \in \omega$
witnesses that $A_1$ is hyperimmune.

We now prove that if $A_0$ and~$A_1$ are $X$-hyperimmune and $A_0$ is $X$-co-c.e., then
the pair $A_0, A_1$ is dependently $X$-hyperimmune.
Let~$\varphi(U, V)$ be an essential $\Sigma^{0,X}_1$ formula.
Define an $X$-c.e.\ sequence of sets $F_0 < F_1 < \dots$ such that for every~$i \in \omega$,
there is some $R < F_i$ such that~$\varphi(R, F_i)$ holds and~$R \subseteq \overline{A}_0$.
First, notice that the sequence is $X$-c.e.\ since $A_0$ is $X$-co-c.e.
Second, we claim that the sequence is infinite. To see this,
define an $X$-c.e.\ array $E_0 < E_1 < \dots$ such that for every~$i \in \omega$, there is some finite set~$S > E_i$
such that $\psi(E_i, S)$ holds. The array is infinite since $\psi(U, V)$ is essential.
Since $A_0$ is $X$-hyperimmune, there are infinitely many $i$'s such that $E_i \subseteq \overline{A}_0$.
Last, by $X$-hyperimmunity of~$A_1$, there is some~$i \in \omega$ such that~$F_i \subseteq \overline{A}_1$.
By definition of~$F_i$, there is some~$R \subseteq \overline{A}_0$ such that~$\varphi(R, F_i)$ holds. \qed
\end{proof}

\begin{corollary}
$\rt^2_2$ admits preservation of dependent hyperimmunity for co-c.e.\ sets.
\end{corollary}
%\begin{proof}
%By Corollary~\ref{cor:rt22-preservation-hyperimmunity-coce} 
%and Lemma~\ref{lem:dependent-hyperimmunity-for-coce}. \qed
%\end{proof}

%% file: parts/part5-wscac-scac.tex
In their seminal paper~\cite{Cholak2001strength}, Cholak, Jockusch and Slaman had the idea to split Ramsey's theorem for pairs
into two simpler statements, namely, stable Ramsey's theorem for pairs ($\srt^2_2$) and cohesiveness ($\coh$),
in order to get more insights about the strength of $\rt^2_2$ by manipulating $\srt^2_2$ and $\coh$ independently.
Hirschfeldt and Shore~\cite{Hirschfeldt2007Combinatorial} applied the same idea to their statements about linear and partial orders,
and introduced the corresponding notions of stability. In the case of partial orders, there are however
two possible notions of stability.

Given a partial order $(P, \leq_P)$, we say that $x \in P$ is \emph{small}, \emph{large} or \emph{isolated}
if for all but finitely many $y \in P$, $x \leq_P y$, $x \geq_P y$, or $x |_P y$, respectively.
We write $S^{*}(P)$, $L^{*}(P)$ and $I^{*}(P)$ for the set of small, large and isolated elements of $P$, respectively.
A partial order is \emph{weakly stable} if every element is either small, large, or isolated, that is,
$P = S^{*}(P) \cup L^{*}(P) \cup I^{*}(P)$.
A partial order is \emph{stable} if every element is small or isolated, or if every element is large or isolated,
that is, $P = S^{*}(P) \cup I^{*}(P)$ or $P = L^{*}(P) \cup I^{*}(P)$. We let $\scac$ and $\wscac$
be the restriction of $\cac$ to stable and weakly stable partial orders, respectively.

The notion of stable partial order was introduced by Hirschfeldt and Shore.
They proved that $\ads$ is equivalent to the statement ``Every partial order has a stable suborder'',
showing therefore that $\rca \vdash \cac \biimp \ads \wedge \scac$.
Hirschfeldt and Shore noticed that $\scac$ was not the most immediate notion of stability,
but justified their choice by proving an equivalence between $\scac$ and the restriction of $\srt^2_2$
to semi-transitive colorings. Furthermore, Jockusch et al.~\cite{Jockusch2009Stability} proved that $\rca \vdash \scac \biimp \wscac$,
showing therefore that the choice of either notion had no impact to the strength of the statement
in reverse mathematics.
However, the implication $\scac \imp \wscac$ involved two applications of $\scac$, and 
Astor et al.~\cite{Astor2016uniform} proved that those two applications were necessary, by showing that $\wscac \not \leq_c \scac$.
Their proof uses a very involved notion of forcing building an instance of $\wscac$ and solutions to $\scac$ simultaneously.
In this section, we give a simpler proof formulated in terms of immunity, 
and furthermore show that the instance of $\wscac$ can be chosen 
to be computable.

\begin{definition}[Combined immunity]
A formula~$\varphi(u, v)$ is \emph{combinedly essential}
if for every $x \in \omega$, there are some $r, s > x$ such that~$r \neq s$
and $\varphi(r, s)$ holds.
A pair of sets~$A_0, A_1 \subseteq \omega$ is \emph{combinedly $X$-immune} 
if for every combinedly essential $\Sigma^{0,X}_1$ formula~$\varphi(u, v)$,
$\varphi(r, s)$ holds for some~$r \in \overline{A}_0$ and~$s \in \overline{A}_1$.
\end{definition}

In particular, if $A_0$ and $A_1$ are combinedly $X$-immune, 
then so are $A_1$ and $A_0$. Moreover, they are both $X$-immune.
Note that the notion of combined immunity differs from dependent immunity
by the alternation of quantifiers in the definition of essentiality.

\begin{theorem}\label{thm:wscac-combinedly-immune}
There is a computable weakly stable partial order $(P, \leq_P)$
such that $S^{*}(P) \cup L^{*}(P)$ and $I^{*}(P)$ are both hyperimmune,
and are combinedly immune.
\end{theorem}
\begin{proof}
Fix an enumeration~$\varphi_0(U), \varphi_1(U), \dots$ of all $\Sigma^0_1$ formulas where $U$ is a finite coded set parameter,
and an enumeration $\psi_0(u, v), \psi_1(u, v), \dots$ of all $\Sigma^0_1$ formulas where $u$ and $v$ are integer parameters.
The construction of the partial order $(P, \leq_P)$ is done by a finite injury priority argument with a movable marker procedure. 
Recall that a formula $\varphi(U)$ where $U$ is a finite coded set is \emph{essential}
if for every $x \in \omega$, there is some finite set~$R > x$ such that $\varphi(R)$ holds.
The following schemes of requirements ensure that $I^{*}(P)$ and $S^{*}(P) \cup L^{*}(P)$
will both be hyperimmune.
\[
\Rcal_e : \varphi_e(U) \mbox{ essential} \imp 
	(\exists R \subseteq_{fin} S^{*}(P) \cup L^{*}(P))\varphi_e(R)
\]
\[
\Scal_e : \varphi_e(U) \mbox{ essential} \imp 
	(\exists R \subseteq_{fin} I^{*}(P))\varphi_e(R) \hspace{1.2cm}
\]
The following scheme of requirements ensures that $S^{*}(P) \cup L^{*}(P)$ and $I^{*}(P)$
are combinedly immune.
\[
\Tcal_e : \psi_e(u, v) \mbox{ combinedly essential} \imp 
	(\exists r \in I^{*}(P))(\exists s \in S^{*}(P) \cup L^{*}(P))\psi_e(r,s)
\]

The requirements are given an interleaved priority ordering.
We proceed by stages, maintaining three sets $S$, $L$ and $I$,
which represent $S^{*}(P)$, $L^{*}(P)$ and $I^{*}(P)$, respectively.
At stage 0, $S_0 = L_0 = \emptyset$ and $I_0 = \{0\}$ and $\leq_P$ is nowhere defined.
Moreover, to each group of requirements $\Rcal_e$, $\Scal_e$, $\Tcal_e$, we associate
a marker $m_e$, initialized to~$0$.

A strategy for~$\Rcal_e$ \emph{requires attention at stage~$s+1$}
if $\varphi_e(R)$ holds for some~$R \subseteq [m_e, s]$.
The strategy sets~$S_{s+1} = S_s \cup [m_e, s]$, $L_{s+1} = L_s \setminus [m_e, s]$
and $I_{s+1} = I_s \setminus [m_e, s]$.

A strategy for~$\Scal_e$ \emph{requires attention at stage~$s+1$}
if $\varphi_e(R)$ holds for some~$R \subseteq [m_e, s]$.
The strategy sets~$S_{s+1} = S_s \setminus [m_e, s]$, $L_{s+1} = L_s \setminus [m_e, s]$,
and $I_{s+1} = I_s \cup [m_e, s]$.

A strategy for~$\Tcal_e$ \emph{requires attention at stage~$s+1$}
if $\varphi_e(u, v)$ holds for some~$u, v \in [m_e, s]$ such that $u \neq v$.
Let $\{v\} \uparrow = \{x \in [m_e, s] : v \leq_P x \}$ and $\{v\} \downarrow = \{x \in [m_e, s] : v \geq_P x \}$.
If $u \leq_P v$, then the strategy sets $S_{s+1} = S_s \setminus [m_e, s]$,
$L_{s+1} = (L_s \setminus [m_e, s]) \cup \{v\} \uparrow$, and $I_{s+1} = I_s \cup ([m_e, s] \setminus \{v\} \uparrow)$.
If $u \not \leq_P v$, then the strategy sets $S_{s+1} = (S_s \setminus [m_e, s]) \cup \{v\} \downarrow$,
$L_{s+1} = L_s \setminus [m_e, s]$, and $I_{s+1} = I_s \cup ([m_e, s] \setminus \{v\} \downarrow)$.

At stage~$s+1$, assume that~$S_s \cup L_s \cup I_{s} = [0,s]$
and that~$\leq_P$ is defined for each pair over~$[0,s)$.
For each~$x \in [0,s)$, set~$x \leq_P s$ if $x \in S_s$, $x \geq_P s$ if $x \in L_s$, and $x |_P s$ if $I_s$.
If some strategy $\Rcal_e$, $\Scal_e$ or $\Tcal_e$ requires attention at stage~$s+1$, take the least such one and execute it.
Then, declare the strategy satisfied, declare all the strategies of lower priority unsatisfied,
and set $m_i = s+1$ for every $i \geq e$.
If no strategy requires attention, then add $\{s\}$ to~$I_{s+1}$ and go to the next stage. This ends the construction.

Each time a strategy acts, it changes the marker of all strategies of lower priority, and is declared satisfied.
Once a strategy is satisfied, only a strategy of higher priority can injure it.
Therefore, each strategy acts finitely often, and the markers stabilize.
It follows that relation $\leq_P$ is weakly stable.

\begin{claim}
The relation $\leq_P$ is transitive.
\end{claim}
\begin{proof*}
Suppose that there are three elements $x, y <_\Nb z$ such that
$x \leq_P y \leq_P z \leq_P x$. We have two cases.
In the first case, $x <_\Nb y$. Then, by construction, $x \in S_y \cap L_z$ and $y \in S_z$.
Then, $x$ must have been moved to $L$ at a stage $s$ between stage $y$ and stage $z$, 
otherwise we would have $x \geq_P y$ or $x \not \in L_z$. Let $s_0$ be the last such stage.
When $x$ is moved to $L$ at stage $s_0$, it is because of some requirement $\Rcal_e$, $\Scal_e$ or $\Tcal_e$
such that $m_e \leq_\Nb x$. But then, after this, $m_i$ is moved to a value greater than $s_0 \geq y$ for every $i \geq e$.
By construction, when $x$ is moved to $L$ after stage $y$, then so is $\{x\} \uparrow$, and in particular so is $y$. 
Therefore, if $y \in S_z$, it must have been moved out of $L$ at some later stage $s_1 \geq s_0$,
and by a strategy of higher priority $\Rcal_i$, $\Scal_i$ or $\Tcal_i$, for some $i < e$.
By construction, at stage $s_0$, $m_i \leq_\Nb m_e \leq_\Nb x$.
We claim that $m_i \leq_\Nb x$ at stage $s_1$. If not, then 
the value of $m_i$ must have been changed at a stage between $s_0$ and $s_1$, 
but then by construction, $m_i$ has been moved to a value greater than $s_0 \geq y$.
Since $\Rcal_i$, $\Scal_i$ or $\Tcal_i$ can change only values greater than $m_i$, and since $y$ have been changed, we obtain a contradiction. The movement of $y$ can be due to $\Rcal_i$, in which case
$x \in [m_i, s_1] \subseteq S_{s_1}$ since $m_i \leq_\Nb x \leq_\Nb s_1$,
or it can be due to $\Tcal_i$, in which case $y \in \{v\}\downarrow$
for some $v$, and hence $x \in \{y\} \downarrow \subseteq \{v\}\downarrow \subseteq S_{s_1}$.
In both cases, $x \in S_{s_1}$.
By maximality of $s_0$, $x$ does not enter again in $L$ before stage $z$, contradicting $x \in L_z$.
The case where $x >_\Nb y$ is treated similarly.
\end{proof*}

\begin{claim}
For every~$e \in \omega$, $\Rcal_e$, $\Scal_e$ and $\Tcal_e$ are satisfied.
\end{claim}
\begin{proof*}
By induction over the priority order. Fix some~$e$, and let~$s_0$ be a stage after $m_e$ does not change any more.
Suppose for contradiction that $\Rcal_e$ is not satisfied at stage~$s_0$.
Then $\varphi_e(U)$ is essential, so $\varphi_e(R)$ holds for some set~$R > m_e$.
The strategy $\Rcal_e$ will require attention at some stage before $\max(s_0, R)$,
and will receive attention since no strategy of higher priority acts after stage $s_0$.
Then it will act, and $m_e$ will be moved to a value greater than $s_0$, contradiction.
The cases of $\Scal_e$ and $\Tcal_e$ is similar.
\end{proof*}

This last claim finishes the proof. \qed
\end{proof}

\begin{theorem}\label{thm:scac-combined-immunity}
For every computable stable partial order $(P, \leq_P)$ and for every
pair of sets~$A_0, A_1$ which are both hyperimmune, and are combinedly immune,
there is an infinite chain or antichain $H$ such that $A_0$ and $A_1$ are both $H$-immune.
\end{theorem}
\begin{proof}
Fix a computable partial stable partial order $(P, \leq_P)$, and assume that $P = S^{*}(P) \cup I^{*}(P)$.
The other case is symmetric. Assume that there is no chain or antichain $H$ such that
$A_0$ and $A_1$ are both $H$-immune, otherwise we are done.
We will build two infinite sets~$G_0$ and~$G_1$, such that~$G_0$ is an infinite ascending sequence,
$G_1$ is an infinite antichain, and such that $A_0$ and $A_1$ are both $G_i$-immune
for some~$i < 2$.

The construction is done by a variant of Mathias forcing $(F_0, F_1)$, 
where~$F_0 \subseteq S^{*}(P)$ is a finite ascending sequence and $F_1 \subseteq I^{*}(P)$
is a finite antichain. Given a condition $c = (F_0, F_1)$, we let 
$$
X(c) = \{ z \in \omega : (\forall x \in F_0)[x < z \wedge x <_P z] \wedge (\forall y \in F_1)[y < z \wedge y |_P z] \}
$$
Note that the set $X(c)$ is cofinite.
A condition~$d = (E_0, E_1)$ \emph{extends} $c = (F_0, F_1)$ if 
$(E_i, X(d))$ Mathias extends $(F_i, X(c))$ for each~$i < 2$.
A pair of sets~$G_0, G_1$ \emph{satisfies} a condition~$c = (F_0, F_1)$
if~$G_0$ is an infinite ascending sequence, $G_1$ is an infinite antichain,
and $G_i$ satisfies the Mathias condition $(F_i, X(c))$ for each~$i < 2$.

\begin{lemma}\label{lem:scac-combined-immunity-force-Q}
For every condition~$c = (F_0, F_1)$, there is an extension $(E_0, E_1)$
of~$c$ such that~$|E_i| > |F_i|$ for each~$i < 2$.
\end{lemma}
\begin{proof*}
By assumption, $L^{*}(P) \cap X(c)$ and $I^{*}(P) \cap X(c)$ are both infinite, otherwise
there would be a computable infinite chain or antichain.
Take any $x \in L^{*}(P) \cap X(c)$ and $y \in I^{*}(P) \cap X(c)$. 
The condition $(F_0 \cup \{x\}, F_1 \cup \{y\})$ is the desired extension.
\end{proof*}

A condition~$c$ \emph{forces} a formula~$\varphi(G_0, G_1)$ if
$\varphi(G_0, G_1)$ holds for every pair of sets~$G_0, G_1$ satisfying~$c$.
We want to satisfy the disjunctive requirement $\Rcal_{e_0, e_1, i_0, i_1} = \Rcal_{e_0, i_0}^{G_0} \vee \Rcal_{e_1, i_1}^{G_1}$
for each~$e_0, e_1 \in \omega$ and $i_0, i_1 \in \{0, 1\}$, where
$$
\Rcal_{e,i}^G \mbox{ : } \Phi_e^G \mbox{ is not an infinite subset of } A_i
$$
Note that to obtain the desired property, the pair $(i_0, i_1)$ must range over $\{(0,0)$, $(0,1)$, $(1,0), (1,1) \}$.
For this, we are going to use the fact that $(A_0, A_0)$, $(A_0, A_1)$, $(A_1, A_0)$, and $(A_1, A_1)$
are all combinedly immune, which is a consequence of the facts that $(A_0, A_1)$ are combinedly immune
and that both $A_0$ and $A_1$ are hyperimmune.

\begin{lemma}\label{lem:scac-combined-immunity-force-R}
For every condition~$c$, every pair of indices $e_0, e_1 \in \omega$ and every~$i_0, i_1 \in \{0,1\}$,
there is an extension~$d$ of~$c$ forcing~$\Rcal_{e_0, e_1, i_0, i_1}$.
\end{lemma}
\begin{proof*}
Fix a condition~$c = (F_0, F_1)$.
A \emph{split pair} is a pair of finite sets $E_0, E_1 \subseteq X(c)$
such that $F_0 \cup E_0$ is a finite ascending sequence, $F_1 \cup E_1$ is a finite antichain,
and $\max E_0 \leq_P x$ for each~$x \in E_1$.
Let $\varphi(u, v)$ be the $\Sigma^0_1$ formula which holds if there is 
a split pair $E_0, E_1$, such that 
$\Phi_{e_0}^{F_0 \cup E_0}(u) \downarrow$ and $\Phi_{e_1}^{F_1 \cup E_1}(v) \downarrow$.
We have two cases.

\begin{itemize}
	\item Case 1: the formula $\varphi(u, v)$ is not combinedly essential, say with witness~$x$.
	If there is a pair of infinite sets~$G_0, G_1$ satisfying $c$ and some $r > x$ such that
	$\Phi_{e_0}^{G_0}(r) \downarrow$,
	then let~$E_0 \subseteq G_0$ be such that $F_0 \cup E_0$ is an initial segment of~$G_0$
	for which $\Phi_{e_0}^{F_0 \cup E_0}(r) \downarrow$. In particular, $E_0 \subseteq L^{*}(P) \cap X(c)$
	is an finite increasing sequence, so the condition $d = (F_0 \cup E_0, F_1)$ is an extension
	forcing $\Rcal_{e_1, i_1}^{G_1}$, hence forcing $\Rcal_{e_0, e_1, i_0, i_1}$.
	If there is no such pair of infinite sets~$G_0, G_1$, then the condition~$c$
	already forces $\Rcal_{e_0, i_0}^{G_0}$, hence $\Rcal_{e_0, e_1, i_0, i_1}$.

	\item Case 2: the formula $\varphi(u, v)$ is combinedly essential.
	Note that since the pair $A_0, A_1$ is combinedly immune, so is the pair $A_1, A_0$.
	Moreover, since for each~$i \in \{0, 1\}$, $A_i$ is hyperimmune, the pair $A_i, A_i$
	is combinedly immune. Therefore, for every $i_0, i_1 \in \{0,1\}$,
	the pair $A_{i_0}, A_{i_1}$ is combinedly immune.
	In particular, $\varphi(r, s)$ holds for some $r \in \overline{A}_{i_0}$
	and $s \in \overline{A}_{i_1}$.

	Unfolding the definition of $\varphi(r, s)$, there is a split pair $E_0, E_1$ such that
	$\Phi_{e_0}^{F_0 \cup E_0} \cap \overline{A}_{i_0} \neq \emptyset$
	and $\Phi_{e_1}^{F_1 \cup E_1} \cap \overline{A}_{i_1} \neq \emptyset$.
	Since $\max E_0 <_P x$ for each~$x \in E_1$, either $E_0 \subseteq L^{*}(P)$,
	or $E_1 \subseteq I^{*}(P)$. Therefore, either $d_0 = (F_0 \cup E_0, F_1)$,
	or $d_1 = (F_0, F_1 \cup E_1)$ is a valid extension of $c$.
	In particular, $d_0$ forces $\Rcal^{G_0}_{e_0, i_0}$ and $d_1$ forces
	$\Rcal^{G_1}_{e_1, i_1}$. In both case, there is an extension of $c$
	forcing $\Rcal_{e_0, e_1, i_0, i_1}$.
\end{itemize}
\end{proof*}

Let~$\Fcal = \{c_0, c_1, \dots\}$ be a sufficiently generic filter containing $(\emptyset, \emptyset)$,
where~$c_s = (F_{0,s}, F_{1,s})$. The filter~$\Fcal$ yields a
pair of sets~$G_0, G_1$ defined by~$G_i = \bigcup_s F_{i,s}$.
By Lemma~\ref{lem:scac-combined-immunity-force-Q}, the sets $G_0$ and~$G_1$ are both infinite,
and by Lemma~\ref{lem:scac-combined-immunity-force-R}, the sets $A_0$ and $A_1$ are both $G_i$-immune for some~$i < 2$. 
This completes the proof of Theorem~\ref{thm:scac-combined-immunity}. \qed
\end{proof}

\begin{corollary}
$\wscac \not \leq_c \scac$.
\end{corollary}
\begin{proof}
By Theorem~\ref{thm:wscac-combinedly-immune}, there is a computable, weakly stable partial order $(P, \leq_P)$
such that $S^{*}(P) \cup L^{*}(P)$ and $I^{*}(P)$ are both hyperimmune, and are combinedly immune.
By Theorem~\ref{thm:scac-combined-immunity}, for every computable stable partial order $(Q, \leq_Q)$,
there is an infinite $Q$-chain or $Q$-antichain $H$ such that $S^{*}(P) \cup L^{*}(P)$ and $I^{*}(P)$
are both $H$-immune, and therefore such that $H$ does not compute an infinite $P$-chain or $P$-antichain.
\end{proof}

Accordingly, we say that a linear order $(P, \leq_P)$ is \emph{stable} if it is of order type $\omega + \omega^{*}$,
that is, if $P = S^{*}(P) \cup L^{*}(P)$. We let $\sads$ be the restriction of $\ads$ to stable colorings.
Tennenbaum (see Downey~\cite{Downey1998Computability}) constructed a computable linear order of order type $\omega + \omega^{*}$
with no infinite computable ascending or descending sequence. Downey noticed that the construction could be modified so that
the $\omega$ and $\omega^{*}$ part are both hyperimmune. We now show that this is the case for \emph{every}
computable instance of $\sads$ with no computable solution. It follows that every such instance of $\sads$ is a witness
that the Erd\H{o}s-Moser theorem ($\emo$) does not imply $\sads$ over~$\rca$, since the former has been proven
to admit preservation of hyperimmunity (see~\cite{Lerman2013Separating,Patey2015Iterativea}).

\begin{lemma}
For every computable linear order $(P, \leq_P)$ of order type $\omega + \omega^{*}$
with no computable infinite ascending or descending sequence, $S^{*}(P)$ 
and $L^{*}(P)$ are both hyperimmune.
\end{lemma}
\begin{proof}
We first prove that $S^{*}(P)$ is hyperimmune.
Let $F_0, F_1, \dots$ be an array tracing $S^{*}(P)$.
Then, the set $\{ \min_P F_i : i \in \omega \}$ is an infinite $\vec{F}$-computable subset of $S^{*}(P)$,
and therefore the array is not computable. It follows that $S^{*}(P)$ is not traced by any computable array.
The case of $L^{*}(P)$ holds by symmetry.
\end{proof}

%% file: parts/part7-questions.tex
The framework of preservation of computability-theoretic properties enables one, among other things,
to separate compound statements in reverse mathematics by analyzing the preservation of such properties
by each statement separately. For example, Wang~\cite{Wang2014Definability} separated $\coh + \wkl + \emo$ from~$\sads$
over~$\rca$ using the preservation of proper $\Delta^0_2$ definitions. 
In a previous version of this paper, we asked the following question.

\begin{question}
Does $\ads + \wkl$ imply any of $\rt^2_2$ or $\cac$ over~$\rca$?
\end{question}

The statements $\ads$ and $\wkl$ have both known not to imply $\rt^2_2$ over~$\rca$, but for very different reasons.
By proving in Section~\ref{sect:cac-cb-immunity} that $\cac$ admits preservation of c.b-immunity, we showed
that $\cac$, and \emph{a fortiori} $\ads$, does not imply any notion of compactness. On the other hand, $\rt^2_2$
implies the diagonally non-computable principle ($\dnr$), which is equivalent to a very weak form of compactness,
namely, the Ramsey-type weak weak K\"onig's lemma (see~\cite{Bienvenu2015logical,Flood2014Separating}).

Weak K\"onig's lemma, as for him, does not imply $\rt^2_2$ over~$\rca$ for various reasons.
By the low basis theorem~\cite{Jockusch197201}, one can build a model of~$\wkl$ containing only low sets,
while Jockusch~\cite{Jockusch1972Ramseys} constructed a computable instance of $\rt^2_2$ with no $\Delta^0_2$ solution.
One can also separate $\wkl$ from $\rt^2_2$ thanks the the hyperimmune-free basis theorem~\cite{Jockusch197201}.
Indeed, $\rt^2_2$ has a computable instance whose solutions are of hyperimmune degree. Both separations
cannot be adapted to separate $\wkl+\ads$ from $\rt^2_2$ since by Hirshchfeldt and Shore~\cite{Hirschfeldt2007Combinatorial},
$\rca \vdash \ads \imp \coh$ and therefore $\ads$ has a computable instance with no low solution
and whose solutions are of hyperimmune degree.

The question was recently answered negatively by Towsner~\cite{Towsner2016Constructing}, who developed an involved
technique for separating statements about partial orders in presence of weak K\"onig's lemma.

\begin{theorem}[Towsner~\cite{Towsner2016Constructing}]
$\ads + \wkl$ does not imply $\cac$,
and $\cac + \wkl$ does not imply $\rt^2_2$ over $\rca$.
\end{theorem}